\documentclass[a4paper,11pt]{article}

\usepackage[margin=2.5cm]{geometry}

\usepackage{amsmath,amsthm,amssymb,amsfonts,mathrsfs,graphicx,color,bbm,bm}
\usepackage{stmaryrd}
\usepackage{enumitem}

\usepackage[usenames,dvipsnames]{xcolor}
\usepackage[english]{babel}
\usepackage[bookmarksopen,hidelinks,destlabel]{hyperref}

\newcommand{\E}{\mathbb{E}}
\newcommand{\R}{\mathbb{R}}

\newcommand{\N}{\mathbb{N}}
\newcommand{\Z}{\mathbb{Z}}
\renewcommand{\P}{\mathbb{P}}

\newcommand{\condparentheses}[2]{\left(\left.#1\,\right\vert#2\right)}

\newcommand{\condP}[2]{\mathbb{P}\condparentheses{#1}{#2}}
\newcommand{\condE}[2]{\mathbb{E}\condparentheses{#1}{#2}}

\DeclareMathOperator{\Var}{Var}
\DeclareMathOperator{\Cov}{Cov}

\newcommand{\set}[1]{\left\{#1\right\}}

\newcommand{\cointerval}[1]{\left[#1\right)} 

\newcommand{\mat}[1]{\begin{pmatrix}#1\end{pmatrix}}
\newcommand{\smallmat}[1]{\left(\begin{smallmatrix}#1\end{smallmatrix}\right)}

\newcommand{\equalsd}{\overset{d}{=}}

\newcommand{\decreasesto}{\searrow}
\newcommand{\increasesto}{\nearrow}

\renewcommand{\phi}{\varphi}

\let\epsilon\varepsilon

\newcommand{\Binomial}{\mathrm{Binomial}}

\newtheorem{theorem}{Theorem}
\newtheorem{prop}[theorem]{Proposition}
\newtheorem{lemma}[theorem]{Lemma}

\theoremstyle{definition}
\newtheorem{definition}[theorem]{Definition}

\newtheorem{remark}[theorem]{Remark}

\newtheorem{example}[theorem]{Example}

\newcommand{\textandreference}[2]{\texorpdfstring{\hyperref[#2]{#1\ref*{#2}}}{#1\ref*{#2}}}

\newcommand{\lbsect}[1]{\label{s:#1}}

\newcommand{\refsubsect}[1]{\textandreference{Section~}{ss:#1}}

\newcommand{\lbthm}[1]{\label{T:#1}}
\newcommand{\refthm}[1]{\textandreference{Theorem~}{T:#1}}

\newcommand{\lbprop}[1]{\label{P:#1}}
\newcommand{\refprop}[1]{\textandreference{Proposition~}{P:#1}}

\newcommand{\lblemma}[1]{\label{L:#1}}
\newcommand{\reflemma}[1]{\textandreference{Lemma~}{L:#1}}

\newcommand{\lbdefn}[1]{\label{d:#1}}
\newcommand{\refdefn}[1]{\textandreference{Definition~}{d:#1}}

\newcommand{\lbexample}[1]{\label{ex:#1}}
\newcommand{\refexample}[1]{\textandreference{Example~}{ex:#1}}

\newcommand{\refitem}[1]{\ref{item:#1}}

\numberwithin{equation}{section}

\usepackage[numbers]{natbib}

\newcommand{\collapsedrcp}[1]{\vec{#1}}

\newcommand{\relent}[2]{H\!\condparentheses{#1}{#2}}
\newcommand{\bigrelent}[2]{\relent{\big.\smash{#1}}{#2}}
\newcommand{\HXsX}{\bigrelent{X^{(s)}}{X}}

\let\grad\nabla

\begin{document}

\title{The saddlepoint approximation factors over sample paths of recursively compounded processes}
\author{Jesse Goodman\footnote{Department of Statistics, University of Auckland, Private Bag 92019, Auckland 1142, New Zealand}}
\date{\today}
  
\maketitle
  
\begin{abstract}
  This paper presents an identity between the multivariate and univariate saddlepoint approximations applied to sample path probabilities for a certain class of stochastic processes.
  This class, which we term the \emph{recursively compounded processes}, includes branching processes and other models featuring sums of a random number of i.i.d.\ terms; and compound Poisson processes and other L\'evy processes in which the additive parameter is itself chosen randomly.
  For such processes,
  $$
  \hat{f}_{X_1,\dotsc,X_N \,\vert\, X_0=x_0}(x_1,\dotsc,x_N) = \prod_{n=1}^N \hat{f}_{X_n \,\vert\, X_0=x_0,\dotsc,X_{n-1}=x_{n-1}}(x_n)
  ,
  $$
  where the left-hand side is a multivariate saddlepoint approximation applied to the random vector $(X_1,\dotsc,X_N)$ and the right-hand side is a product of univariate saddlepoint approximations applied to the conditional one-step distributions given the past.
  
  Two proofs are given.
  The first proof is analytic, based on a change-of-variables identity linking the functions that arise in the respective saddlepoint approximations.
  The second proof is probabilistic, based on a representation of the saddlepoint approximation in terms of tilted distributions, changes of measure, and relative entropies.
\end{abstract}

\section{Introduction}\lbsect{Intro}

The saddlepoint approximation provides a systematic and explicit method for approximating the unknown probability density function $f(x)$ of a random variable $X$ whose moment generating function is known.
Namely, define 
\begin{equation}\label{BasicUnivariateMGFCGF}
  M_X(s) = \E(e^{sX}), \qquad K_X(s)=\log M_X(s),
\end{equation}
the \emph{moment generating function} (MGF) and \emph{cumulant generating function} (CGF), respectively.
For given $x$, we seek a solution $\hat{s}$ to the \emph{saddlepoint equation}
\begin{equation}\label{UnivariateSE}
  K'_X(\hat{s}) = x
  .
\end{equation}
If such a solution exists, we call $\hat{s}$ the \emph{saddlepoint} and we define
\begin{equation}\label{UnivariateSPA}
  \hat{f}_X(x) = \frac{\exp\left( K_X(\hat{s}) - \hat{s}x \right)}{\sqrt{2\pi K_X''(\hat{s})}}
  ,
\end{equation}
the \emph{saddlepoint approximation} to the probability density function of $X$.

The saddlepoint approximation has a straightforward generalisation to the multivariate case.
The multivariate MGF and CGF for $X_1,\dotsc,X_N$ are
\begin{equation}\label{MultivariateMGFCGF}
  \begin{aligned}
    M_{X_1,\dotsc,X_N}(s_1,\dotsc,s_N) &= \E(e^{s_1 X_1 + \dotsb + s_N X_N}), 
    \\
    K_{X_1,\dotsc,X_N}(s_1,\dotsc,s_N) &= \log M_{X_1,\dotsc,X_N}(s_1,\dotsc,s_N),
  \end{aligned}
\end{equation}
and we write $K'_{X_1,\dotsc,X_N}(s_1,\dotsc,s_N)$, $K''_{X_1,\dotsc,X_N}(s_1,\dotsc,s_N)$ for the gradient (column) vector and Hessian matrix, respectively, both of which are defined everywhere in the interior of the domain.
The saddlepoint equation for $\hat{s}_1,\dotsc,\hat{s}_N$ reads
\begin{equation}\label{MultivariateSE}
  K'_{X_1,\dotsc,X_N}(\hat{s}_1,\dotsc,\hat{s}_N) = \mat{x_1\\ \vdots \\ x_N}
\end{equation}
and, if a solution exists, the saddlepoint approximation becomes
\begin{equation}\label{MultivariateSPA}
  \hat{f}_{X_1,\dotsc,X_N}(x_1,\dotsc,x_N) = \frac{\exp\bigl( K_{X_1,\dotsc,X_N}(\hat{s}_1,\dotsc,\hat{s}_N) - \sum_{n=1}^N \hat{s}_n x_n \bigr)}{\sqrt{\det\bigl( 2\pi K''_{X_1,\dotsc,X_N}(\hat{s}_1,\dotsc,\hat{s}_N) \bigr)}}
  .
\end{equation}

The saddlepoint approximation can be applied equally to integer-valued random variables, with $\hat{f}_X(x)$ yielding an approximation to the probability mass function $f_X(x)=\P(X=x)$ rather than the probability density function.

More generally, the saddlepoint approximation can be applied in various ways to conditional distributions, marginal distributions, and so on, with the choice often determined in practice by which distributions have readily computable CGFs.
See \cite{Butler2007} for a range of examples.
In this paper, we consider a class of processes for which we can compute both a multivariate CGF and CGFs for certain conditional distributions, yielding two different ways in which we can apply the saddlepoint approximation.
We illustrate with an example.

\begin{example}[Galton-Watson branching process]\lbexample{GWBP}
  Let $(X_n)_{n\in\Z_+}$ be a Galton-Watson branching process defined recursively by
  \begin{equation}
    X_n = \sum_{j=1}^{X_{n-1}} Y^{(j)}_n, \quad n\in\N,
  \end{equation}
  where $X_0=x_0\in\N$ is fixed and $Y^{(j)}_n$, $j,n\in\N$, are i.i.d.\ copies of a random variable $Y$ defining the offspring distribution.
  The offspring distribution is typically selected in a simple way with a known CGF $K_Y(s)$.
  
  For fixed $x_1,\dotsc,x_N$, the sample path probability $\condP{X_1=x_1,\dotsc,X_N=x_N}{X_0=x_0}$ factors using the Markov property as 
  \begin{equation}\label{GWBPMarkov}
    f_{X_1,\dotsc,X_N \,\vert\, X_0=x_0}(x_1,\dotsc,x_N) = \prod_{n=1}^N f_{X_n \,\vert\, X_{n-1}=x_{n-1}}(x_n)
    .
  \end{equation}
  Given $X_{n-1}=x_{n-1}$, the conditional distribution of $X_n$ is that of a sum of $x_{n-1}$ i.i.d.\ copies of the offspring distribution $Y$.
  The corresponding CGF $K_{X_n\,\vert\,X_{n-1}=x_{n-1}}(s) = x_{n-1} K_Y(s)$ is known, so we may use $N$ univariate saddlepoint approximations to obtain separate approximations $\hat{f}_{X_n \,\vert\, X_{n-1}=x_{n-1}}(x_n)$ to each transition probability $\condP{X_n=x_n}{X_{n-1}=x_{n-1}}$ in the product from \eqref{GWBPMarkov}.
  
  Alternatively, we may deal directly with the correlated random vector $(X_1,\dotsc,X_N)$.
  The multivariate CGF has an explicit, if slightly complicated, form:
  \begin{equation}
    K_{X_1,\dotsc,X_N \,\vert\, X_0=x_0}(s_1,\dotsc,s_N) = x_0 K_Y(s_1+K_Y(s_2 + K_Y(\dotsb + K_Y(s_{N-1} + K_Y(s_N)) \dotsb)))
    .
  \end{equation}
  We may therefore use a single multivariate saddlepoint approximation to obtain a direct approximation $\hat{f}_{X_1,\dotsc,X_N \,\vert\, X_0=x_0}(x_1,\dotsc,x_N)$ to the sample path probability in \eqref{GWBPMarkov}.
  
  The main result of this paper will show that for this branching process, and for other processes with a similar recursive structure, these two seemingly different approximations are equal.
\end{example}

\section{Main result}\lbsect{MainResult}

\subsection{Recursively compounded processes}

We consider random processes featuring \emph{compound sums} formed from a random number of i.i.d.\ summands, similar to \refexample{GWBP}.
The general definition allows for additional independent terms, time inhomogeneity, dependence on the entire history, non-integer values, and random variables of differing dimensions.

\begin{definition}\lbdefn{RecComp}
  We call a process $(X_0,X_1,\dotsc,X_N)$ \emph{recursively compounded} if, for each $n=1,\dotsc,N$, the conditional distribution of $X_n$ given $X_0,\dotsc,X_{n-1}$ is the same as $Q_n(X_0,\dotsc,X_{n-1})$, where
  \begin{equation}
    Q_n(x_0,\dotsc,x_{n-1}) = \xi_n + \sum_{m=0}^{n-1} \sum_{i=1}^{d_m} Z_{n,m,i}(x_{m,i});
  \end{equation}
  where $\xi_n$ is a random variable and $Z_{n,m,i}(t)$ are processes, all with values in $\R^{d_n}$ and all independent of each other and of $X_0,\dotsc,X_{n-1}$; and where, for each $m<n$ and each $i\leq d_m$, one or more of the following conditions holds:
  \begin{enumerate}
    \item\label{item:RecCompDisc}
    $X_{m,i}$ takes values in $\Z_+$ and $Z_{n,m,i}(t)$ is defined for $t\in\Z_+$ by
    \begin{equation}\label{RecCompDiscFormula}
      Z_{n,m,i}(t) = \sum_{j=1}^t Y_{n,m,i,j},
    \end{equation}
    where $Y_{n,m,i,j}$, $j\in\N$, are i.i.d.\ with values in $\R^{d_n}$, independent of everything else; 
    
    \item\label{item:RecCompCts}
    $X_{m,i}$ takes values in $\cointerval{0,\infty}$ and $(Z_{n,m,i}(t))_{t\in\cointerval{0,\infty}}$ is a L\'evy process with values in $\R^{d_n}$, independent of everything else; or
    
    \item\label{item:RecCompLin}
    $X_{m,i}$ has arbitrary values and $Z_{n,m,i}(t)=ct$ for some constant $c\in\R^{d_n}$.
  \end{enumerate}
\end{definition}

Note that in \eqref{RecCompDiscFormula} and elsewhere, sums $\sum_{j=n}^m a_j$ are taken to be identically 0 when $m=n-1$.
The multivariate MGF $M_{X_1,\dotsc,X_N}(s_1,\dotsc,s_N)$ and the resulting saddlepoint approximation have a simple extension to allow each $X_n$ to be vector-valued, see the discussion in \refsubsect{VectorsGradients}.

\subsection{Main result: sample path probabilities}

The main result of this paper is that, applied to a sample path probability for a recursively compounded process, the saddlepoint approximation can be factored across steps.

\begin{theorem}\lbthm{SamplePathProbsFactor}
  Let $(X_0,X_1,\dotsc,X_N)$ be a recursively compounded process and let $x_n\in\R^{d_n}$ for all $n$.
  Then
  \begin{equation}\label{SamplePathProbsAsProduct}
    \hat{f}_{X_1,\dotsc,X_N \,\vert\, X_0=x_0}(x_1,\dotsc,x_N) = \prod_{n=1}^N \hat{f}_{X_n \,\vert\, X_0=x_0,\dotsc,X_{n-1}=x_{n-1}}(x_n)
    .
  \end{equation}
\end{theorem}

In \eqref{SamplePathProbsAsProduct}, each factor on the right-hand side is a saddlepoint approximation applied to the conditional distribution of $X_n$ given $X_m=x_m$ for $m<n$.
In the notation of \refdefn{RecComp}, 
\begin{equation}
  \hat{f}_{X_n \,\vert\, X_0=x_0,\dotsc,X_{n-1}=x_{n-1}}(x_n) = \hat{f}_{Q_n(x_0,\dotsc,x_{n-1})}(x_n),
\end{equation}
the saddlepoint approximation evaluated at $x_n$ for the random variable $Q_n(x_0,\dotsc,x_{n-1})$.

\begin{remark}
  Part of the assertion of \refthm{SamplePathProbsFactor} is that the two sides of \eqref{SamplePathProbsAsProduct} have the same domain.
  That is, if one side is undefined (either because the corresponding saddlepoint $\hat{s}$ does not exist or because the determinant from \eqref{MultivariateSPA} vanishes) then so is the other, and vice versa.
\end{remark}

\subsection{Examples}

\begin{example}
  The general form of a recursively compounded process with non-negative-integer-valued random variables, as in \refdefn{RecComp}~\refitem{RecCompDisc}, could be described as a time-inhomogeneous age-dependent discrete-time branching process with immigration.
  If the process is vector-valued, the branching process is a multitype branching process.
The random variables $\xi_n$ describe the immigration, and the random variables $Y^{(j)}_{n,n-a,i}$ describe the offspring distribution at time $n$ for individuals of type $i$ and current age $a\in\N$.
\end{example}

\begin{example}\lbexample{PedDavFok}
  Pedeli, Davison and Fokianos \cite{PedDavFok2015} model a time series of non-negative-integer valued data as an integer-valued auto-regressive process of order $p$ by recursively setting
  \begin{equation}
    X_n = \xi_n + \Binomial(X_{n-1},\alpha_1) + \dotsb + \Binomial(X_{n-p},\alpha_p)
    ,
  \end{equation}
  where $\xi_n$ are i.i.d.\ innovations and where all terms are chosen conditionally independently at each step.
  This is a recursively compounded process as in \refdefn{RecComp}~\refitem{RecCompDisc}: $Y_{n,m,1}$ has the Bernoulli$(\alpha_{n-m})$ distribution if $n-p\leq m<n$, with $Y_{n,m,1}=0$ for $m<n-p$.
  (Slight adjustments might be needed to handle the first $p$ time series values, which in \cite{PedDavFok2015} are taken as fixed.
  This can be accommodated in the setup of \refdefn{RecComp} by storing $p$ initial values $X_0,X_{-1},\dotsc,X_{1-p}$ as an initial vector $\vec{X}_0$.)
  
  The probabilities $\alpha_1,\dotsc,\alpha_p$ play the role of autoregressive coefficients, to be estimated based on data.
  Due to the integer-valued nature of the model, standard estimators based on Gaussian regression may be unsatisfactory, and maximum-likelihood-based estimators are proposed instead.
  Computing the likelihood means computing the sample path probability $\condP{X_1=x_1,\dotsc,X_N=x_N}{\smash{\vec{X}_0=\vec{x}_0}}$, where $x_1,\dotsc,x_N$ are observed values from a time series (after storing the first $p$ time series values in $\vec{x}_0$).
  Exact calculation becomes impractical once $p$ is even moderately large, but the saddlepoint approximation remains computable, so the authors use \eqref{SamplePathProbsAsProduct} as an approximate likelihood function and derive an estimator from it.
  
  \refthm{SamplePathProbsFactor} shows that this approximate likelihood, and the corresponding estimator, are equivalent whether computed by a multivariate or a step-wise approach, as in the left- and right-hand sides, respectively, of \eqref{SamplePathProbsAsProduct}.
\end{example}

\begin{example}
  Davison, Hautphenne and Kraus \cite{DavHauKra2021} perform inference on birth and death rates for a population modelled by a continuous-time branching process observed at discrete times.
  The branching property implies that this process is a recursively compounded process, and indeed it reduces to a Galton-Watson branching process as in \refexample{GWBP} if the observation times are equally spaced.
  Similar to \refexample{PedDavFok}, the authors use the saddlepoint approximation to obtain approximations to a sample path probability, which is interpreted an approximate likelihood, and then maximise this approximate likelihood to obtain saddlepoint-based estimates for the birth and death rates.
  
  The saddlepoint approximation can be applied either in univariate form to approximate the separate transition probabilities, or in multivariate form to approximate the sample path probabilities.
  The authors implement both approaches \cite[Web Appendix C]{DavHauKra2021}, with the univariate saddlepoint equations solved algebraically and the multivariate saddlepoint equation solved numerically.
  The two approaches yielded similar calculated values, in accordance with \refthm{SamplePathProbsFactor}, though the non-trivial numerical errors suggest the potential delicacy of solving multivariate saddlepoint equations numerically.
\end{example}

\section{Proofs}\lbsect{Proofs}

We give two proofs of \refthm{SamplePathProbsFactor}, one analytic and one probabilistic.
Both proofs rely on the structure of recursively compounded processes as expressed via CGFs.
We begin by stating vector and derivative conventions for multivariate MGFs and other functions.

\subsection{Conventions for vectors, gradients, Jacobians, and Hessians}\label{ss:VectorsGradients}

We allow each $X_n$ to have vector values, $X_n\in\R^{d_n}$, so that the MGF $M_{X_n}(s_n)$ and CGF $K_{X_n}(s_n)$ become functions of a vector-valued argument $s_n$ of the same dimension.
To match with \eqref{BasicUnivariateMGFCGF}, we follow the vector convention of \cite{GoodmanMLEAccuracy2022}: $X_n$ is interpreted as a column vector or $d_n\times 1$ matrix, and $s_n$ is a row vector or $1\times d_n$ matrix, so that $s_n X_n$ is a $1\times 1$ matrix, i.e., a scalar.
With this convention, $M_{X_n}(s_n)=\E(e^{s_n X_n})=\E(e^{s_{n,1}X_{n,1}+\dotsb+s_{n,d_n}X_{n,d_n}})$ is the multivariate MGF of $X_n$, where $X_{n,i}$ and $s_{n,i}$, $i=1,\dotsc,d_n$, denote the entries of the vectors $X_n$ and $s_n$, respectively.
The multivariate MGF $M_{X_1,\dotsc,X_N}(s_1,\dotsc,s_N)$ from \eqref{MultivariateMGFCGF} coincides with the multivariate MGF of the concatenated vector $\vec{X}=\smallmat{X_1\\ \vdots \\ X_N}$ with vector argument $\vec{s}=\mat{s_1 & \dotsb & s_N}$, interpreted as column and row vectors, respectively, both written in block form and having dimension $D=\sum_{n=1}^N d_n$.

Let $\phi(s_1,\dotsc,s_N)$ be a function with $k$-dimensional row vector values depending on row vectors $s_1,\dotsc,s_N$ of dimensions $d_1,\dotsc,d_N$.
By $\phi'(s_1,\dotsc,s_N)$ we mean the $D\times k$ matrix of partial derivatives $\frac{\partial\phi_j}{\partial s_{m,i}}$, with rows indexed by the entries $s_{m,i}$ (ordered lexicographically as $s_{1,1},\dotsc,s_{1,d_1}$; $s_{2,1},\dotsc,s_{2,d_2};\dotsc;s_{N,1},\dotsc,s_{N,d_N}$) and columns indexed by the entries of $\phi$.
Thus when $\phi$ has scalar values, $\phi'$ means the gradient, interpreted as a column vector; when $\phi$ has vector values, $\phi'$ means the Jacobian matrix (or possibly its transpose, depending on one's preferred row and column conventions for Jacobians).
When $\phi$ has scalar values, $\phi''(s_1,\dotsc,s_N)$ means the $D\times D$ Hessian matrix of mixed partial derivatives $\frac{\partial^2\phi}{\partial s_{m,i}\partial s_{n,j}}$.

With these conventions, the saddlepoint approximation uses almost identical notation when applied to vector-valued random variables, except that the denominator of \eqref{UnivariateSPA} is replaced by $\sqrt{\det(2\pi K''_X(\hat{s}))}$ similar to \eqref{MultivariateSPA}.
In the remainder of the paper, $X_n$ is vector-valued of dimension $d_n$, but the reader may restrict their attention to scalar-valued $X_n$ with virtually no change to the text.

\subsection{The multivariate CGF of \texorpdfstring{$X_0,\dotsc,X_N$}{X\_0, ..., X\_N}}

The key feature uniting \refdefn{RecComp}~\ref{item:RecCompDisc}--\ref{item:RecCompLin} is that the CGFs for $Z_{n,m,i}(t)$ depend linearly on $t$.
Indeed, it is easy to see that $Z_{n,m,i}(t+t')$ has the same distribution as the sum of two independent random variables with distributions $Z_{n,m,i}(t)$ and $Z_{n,m,i}(t')$; here we assume that $t,t'$ belong to $\Z_+$, $\cointerval{0,\infty}$, or $\R$, depending on which of \ref{item:RecCompDisc}--\ref{item:RecCompLin} applies.
Thus
\begin{equation}
  \begin{aligned}
    M_{Z_{n,m,i}(t+t')}(s_n) &= M_{Z_{n,m,i}(t)}(s_n) M_{Z_{n,m,i}(t')}(s_n), 
    \\
    K_{Z_{n,m,i}(t+t')}(s_n) &= K_{Z_{n,m,i}(t)}(s_n) + K_{Z_{n,m,i}(t')}(s_n),
  \end{aligned}
\end{equation}
leading to
\begin{equation}\label{KZtKZ1}
  K_{Z_{n,m,i}(t)}(s) = t K_{Z_{n,m,i}(1)}(s)
\end{equation}
for all applicable $t$.

Similarly, the CGF of $Q_n(x_0,\dotsc,x_{n-1})$ is affine in $x_0,\dotsc,x_{n-1}$.
To state this compactly in vector notation, define $g_{n,m}(s_n)$ for $m<n$ to be the $d_m$-dimensional row vector whose $i^\text{th}$ entry $g_{n,m,i}(s_n)$ is given by
\begin{equation}\label{gnmiFormula}
  g_{n,m,i}(s_n) = K_{Z_{n,m,i}(1)}(s_n).
\end{equation}
Then, by \refdefn{RecComp} and \eqref{KZtKZ1}, 
\begin{align}
  K_{Q_n(x_0,\dotsc,x_{n-1})}(s_n) &= K_{\xi_n}(s_n) + \sum_{m=0}^{n-1} \sum_{i=1}^{d_m} K_{Z_{n,m,i}(1)}(s_n) x_{m,i}
  \notag\\&
  = K_{\xi_n}(s_n) + \sum_{m=0}^{n-1} g_{n,m}(s_n) x_m 
  .
  \label{KQnFormula}
\end{align}

\begin{lemma}\lblemma{KX0XN}
  For $m=N,N-1,\dotsc,1,0$, let the functions $\tau_m(s_m,\dotsc,s_N)$ with $d_m$-dimensional row vector values be defined recursively by $\tau_N(s_N) = s_N$ and
  \begin{equation}\label{tauRecursion}
    \tau_m(s_m,\dotsc,s_N) = s_m + \sum_{n=m+1}^N g_{n,m}(\tau_n(s_n,\dotsc,s_N)) \quad\text{for }0\leq m<N
    .
  \end{equation}
  Then
  \begin{align}\label{KX0XNFormula}
    K_{X_0,\dotsc,X_N}(s_0,\dotsc,s_N) = K_{X_0}(\tau_0(s_0,\dotsc,s_N)) + \sum_{n=1}^N K_{\xi_n}(\tau_n(s_n,\dotsc,s_N))
    .
  \end{align}
\end{lemma}

The proof of \reflemma{KX0XN} is conceptually simple, recursively applying the conditional distributions from \refdefn{RecComp} to simplify the MGF and CGF.
The details and notation become slightly intricate and will reappear only in the proof of \reflemma{PreserveRecComp} below.

\begin{proof}
Condition on the past before step $n$ and use \eqref{KQnFormula} to obtain
\begin{align}
  &M_{X_0,\dotsc,X_n}(s_0,\dotsc,s_n) 
  \notag\\&\quad
  = \E\left( \exp\left( \sum_{m=0}^{n-1} s_m X_m \right) \condE{e^{s_n X_n}}{X_0,\dotsc,X_{n-1}} \right)
  \notag\\&\quad
  = \E\left( \exp\left( \sum_{m=0}^{n-1} s_m X_m \right) \left. \exp\left( K_{\xi_n}(s_n) + \sum_{m=0}^{n-1} g_{n,m}(s_n) x_m \right) \right\vert_{x_m=X_m \,\forall m<n} \right)
  \notag\\&\quad
  = \exp\left( K_{\xi_n}(s_n) \right) \E\left( \exp\left( \sum_{m=0}^{n-1} [s_m + g_{n,m}(s_n)] X_m \right) \right)
  .
\end{align}
Taking logarithms of both sides, the final equation can be equivalently written as
\begin{equation}\label{KnKxiKnminus1Relation}
  K_{X_0,\dotsc,X_n}(s_0,\dotsc,s_n) = K_{\xi_n}(s_n) + K_{X_0,\dotsc,X_{n-1}}(s_0+g_{n,0}(s_n), \dotsc, s_{n-1} + g_{n,n-1}(s_n))
  .
\end{equation}
To complete the proof, we show by induction that for $n=N,N-1,\dotsc,1,0$,
\begin{equation}\label{KX0XNInductive}
  K_{X_0,\dotsc,X_N}(s_0,\dotsc,s_N) = K_{X_0,\dotsc,X_n}(\rho_{n,0},\dotsc,\rho_{n,n}) + \sum_{m=n+1}^N K_{\xi_m}(\tau_m(s_m,\dotsc,s_N))
  ,
\end{equation}
where we have abbreviated
\begin{equation}\label{rhonmFormula}
  \rho_{n,m} = \rho_{n,m}(s_m,\dotsc,s_N) = s_m+\sum_{k=n+1}^N g_{k,m}(\tau_k(s_k,\dotsc,s_N)), \quad 0\leq m\leq n \leq N
  ,
\end{equation}
and where for $n=N$ we interpret the empty sums as $0$.
Thus in the base case $n=N$ we have $\rho_{N,m}(s_m,\dotsc,s_N)=s_m$ and the sum in \eqref{KX0XNInductive} vanishes. 
In particular, \eqref{KX0XNInductive} holds trivially for $n=N$.

To advance the induction, suppose \eqref{KX0XNInductive} holds for a given value $n\in\set{1,\dotsc,N}$.
Apply \eqref{KnKxiKnminus1Relation} with $s_m$ replaced by $\rho_{n,m}$ to obtain
\begin{align}
  K_{X_0,\dotsc,X_N}(s_0,\dotsc,s_N) 
  &= K_{\xi_n}(\rho_{n,n}) + \sum_{m=n+1}^N K_{\xi_m}(\tau_m(s_m,\dotsc,s_N))
  \notag\\&\quad
   + K_{X_0,\dotsc,X_{n-1}}\left( \rho_{n,0}+g_{n,0}(\rho_{n,n}), \dotsc, \rho_{n,n-1}+g_{n,n-1}(\rho_{n,n}) \right) 
  .
  \label{KX0XNrho}
\end{align}
Comparing \eqref{tauRecursion} and \eqref{rhonmFormula}, we see that 
\begin{equation}\label{rhonntaun}
  \rho_{n,n}= s_n+\sum_{k=n+1}^N g_{k,n}(\tau_k(s_k,\dotsc,s_N)) = \tau_n(s_n,\dotsc,s_N), 
\end{equation}
so the first term in the right-hand side of \eqref{KX0XNrho} can be merged into the summation as the $m=n$ term.
Similarly
\begin{equation}\label{rhonmRecursion}
  \rho_{n,m} + g_{n,m}(\tau_n(s_n,\dotsc,s_N)) = s_m + \sum_{k=n}^N g_{k,m}(\tau_k(s_{k+1},\dotsc,s_N)) = \rho_{n-1,m},
\end{equation}
so that \eqref{KX0XNrho} reduces to
\begin{equation}
  K_{X_0,\dotsc,X_N}(s_0,\dotsc,s_N) 
  =  \sum_{m=n}^N K_{\xi_m}(\tau_m(s_m,\dotsc,s_N)) 
  + K_{X_0,\dotsc,X_{n-1}}\left( \rho_{n-1,0}, \dotsc, \rho_{n-1,n-1} \right)
  ,
\end{equation}
which is \eqref{KX0XNInductive} with $n$ replaced by $n-1$.
By induction, \eqref{KX0XNInductive} holds for all $0\leq n<N$.
Applying \eqref{rhonntaun} with $n=0$ gives $\rho_{0,0}=\tau_0(s_0,\dotsc,s_N)$, so \eqref{KX0XNInductive} with $n=0$ reduces to \eqref{KX0XNFormula}.
\end{proof}

\subsection{First proof of \refthm{SamplePathProbsFactor}}

Henceforth, fix $x_0$ and $x_1,\dotsc,x_N$, and write $\vec{x}$ for the column vector of dimension $D=\sum_{n=1}^N d_n$ formed by concatenating the column vectors $x_1,\dotsc,x_N$.
For notational convenience, we will abbreviate $Q_1(x_0),\dotsc,Q_N(x_0,\dotsc,x_{N-1})$ as $Q_1,\dotsc,Q_N$.

In \refdefn{RecComp}, the conditional distribution of $(X_1,\dotsc,X_N)$ is specified recursively in terms of $X_0$, which has an arbitrary starting distribution.
Conditioning on $X_0=x_0$ is therefore equivalent to setting $X_0$ to be the constant random variable with value $x_0$ and CGF 
\begin{equation}\label{X0Choice}
  K_{X_0}(s_0)=s_0 x_0
  .
\end{equation}
We assume this choice of $X_0$ henceforth.
Then
\begin{equation}\label{KX1XNKX0XN}
  K_{X_1,\dotsc,X_N\,\vert\,X_0=x_0}(s_1,\dotsc,s_N) = K_{X_0,X_1,\dotsc,X_N}(0,s_1,\dotsc,s_N)
\end{equation}
and the left-hand side of \eqref{SamplePathProbsAsProduct} is the saddlepoint approximation arising from this CGF.
On the other hand, the product in the right-hand side of \eqref{SamplePathProbsAsProduct} can be interpreted as a multivariate saddlepoint approximation
\begin{equation}\label{prodhatfXgivenhatQ}
  \prod_{n=1}^N \hat{f}_{X_n \,\vert\, X_0=x_0,\dotsc,X_{n-1}=x_{n-1}}(x_n) = \hat{f}_{Q_1,\dotsc,Q_N}(x_1,\dotsc,x_N)
\end{equation}
where $Q_1,\dotsc,Q_N$ are independent and have multivariate CGF
\begin{equation}
  K_{Q_1,\dotsc,Q_N}(s_1,\dotsc,s_N) = \sum_{n=1}^N K_{Q_n}(s_n).
\end{equation}
To see this, note that the Hessian matrix $K''_{Q_1,\dotsc,Q_N}(s_1,\dotsc,s_N)$ is block-diagonal with diagonal blocks $K_{Q_n}''(s_n)$, so its determinant factors as a product of the determinants of its diagonal blocks.
Similarly, the gradient $K'_{Q_1,\dotsc,Q_N}(s_1,\dotsc,s_N)$ is the concatenation of the gradients $K_{Q_n}'(s_n)$, so the multivariate saddlepoint equation $K'_{Q_1,\dotsc,Q_N}(\hat{s}_1,\dotsc,\hat{s}_N)=\vec{x}$ is equivalent to the system of $N$ separate saddlepoint equations $K_{Q_n}'(s_n)=x_n$ for $n=1,\dotsc,N$.

The proof of \refthm{SamplePathProbsFactor} will proceed by relating the CGFs arising in these two multivariate saddlepoint approximations.
We begin with additional definitions to shorten and unify notation.

Since we have taken $X_0$ to be deterministic, its effects can be merged with the random variables $\xi_n$ by defining
\begin{equation}\label{tildexinFormula}
  \tilde{\xi}_n=Q_n(x_0,0,\dotsc,0)=\xi_n + \sum_{i=1}^{d_0} Z_{n,0,i}(x_{0,i}) \quad\text{for }n=1,\dotsc,N
  .
\end{equation}
Note that our independence assumptions imply that $\tilde{\xi}_1,\dotsc,\tilde{\xi}_N$ are independent random variables.
By \eqref{KQnFormula},
\begin{align}
  K_{\tilde{\xi}_1,\dotsc,\tilde{\xi}_N}(s_1,\dotsc,s_N) 
  = \sum_{n=1}^N K_{Q_n(x_0,0,\dotsc,0)}(s_n) 
  = \sum_{n=1}^N \left( \big. K_{\xi_n}(s_n) + g_{n,0}(s_n)x_0 \right)
  .
  \label{KtildexiFormula}
\end{align}
Also define
\begin{equation}\label{taunTFormula}
  T(s_1,\dotsc,s_N) = \mat{\tau_1(s_1,\dotsc,s_N) & \tau_2(s_2,\dotsc,s_N) & \dotsb & \tau_N(s_N)}
  ,
\end{equation}
i.e., $T$ is a mapping from $D$-dimensional row vectors to $D$-dimensional row vectors (in block form) whose $m^\text{th}$ block is given by the function $\tau_m(s_m,\dotsc,s_N)$ defined in \reflemma{KX0XN}.
Likewise, define $G(s_1,\dotsc,s_N)$ to be the $D$-dimensional row vector whose $m^\text{th}$ block $G_m(s_1,\dotsc,s_N)$ is the $d_m$-dimensional row vector
\begin{equation}\label{GmFormula}
  G_m(s_1,\dotsc,s_N) = \sum_{n=m+1}^N g_{n,m}(s_n), \qquad G_N(s_1,\dotsc,s_N)=0
  .
\end{equation}

With these definitions, 
\begin{align}
  &K_{X_1,\dotsc,X_N\,\vert\,X_0=x_0}(s_1,\dotsc,s_N) 
  \label{KXgivenX0KxiT}
  \\&\quad
  = K_{X_0}(\tau_0(0,s_1,\dotsc,s_N)) + \sum_{n=1}^N K_{\xi_n}(\tau_n(s_n,\dotsc,s_N)) 
  &&\text{by \eqref{KX1XNKX0XN} and \reflemma{KX0XN}}
  \notag\\&\quad
  = \tau_0(0,s_1,\dotsc,s_N)x_0 + \sum_{n=1}^N K_{\xi_n}(\tau_n(s_n,\dotsc,s_N))
  &&\text{by \eqref{X0Choice}}
  \notag\\&\quad
  = \sum_{n=1}^N \left( \big. K_{\xi_n}(\tau_n(s_n,\dotsc,s_N)) + g_{n,0}(\tau_n(s_n,\dotsc,s_N))x_0 \right)
  &&\text{by \eqref{tauRecursion} for $m=0$}
  \notag\\&\quad
  = K_{\tilde{\xi}_1,\dotsc,\tilde{\xi}_N} \circ T(s_1,\dotsc,s_N)
  &&\text{by \eqref{KtildexiFormula}--\eqref{taunTFormula}}
  \notag
\end{align}
whereas by independence
\begin{align}
  \label{KQ1QNFormula}
  &K_{Q_1,\dotsc,Q_N}(s_1,\dotsc,s_N) = \sum_{n=1}^N K_{Q_n}(s_n)
  \\&\quad
  = \sum_{n=1}^N \left( K_{\xi_n}(s_n) + g_{n,0}(s_n)x_0 + \sum_{m=1}^{n-1} g_{n,m}(s_n)x_m \right)
  &&\text{by \eqref{KQnFormula}}
  \notag\\&\quad
  = \sum_{n=1}^N \left( K_{\xi_n}(s_n) + g_{n,0}(s_n)x_0 \right) + \sum_{m=1}^N \sum_{n=m+1}^N g_{n,m}(s_n)x_m
  &&\text{by interchanging sums}
  \notag\\&\quad
  = K_{\tilde{\xi}_1,\dotsc,\tilde{\xi}_N}(s_1,\dotsc,s_N) + G(s_1,\dotsc,s_N)\vec{x}
  &&\text{by \eqref{KtildexiFormula} and \eqref{GmFormula}}
  \notag
  .
\end{align}
Moreover \eqref{tauRecursion}, the recursive definition of the functions $\tau_1,\dotsc,\tau_N$, can be expressed as the vector identity
\begin{equation}\label{TGIdentity}
  T(s_1,\dotsc,s_N) = \mat{s_1 & \dotsb & s_N} + G\circ T(s_1,\dotsc,s_N)
  .
\end{equation}

\begin{lemma}\lblemma{FunctionsMappedByT}
  The functions 
  \begin{equation}\label{TwoFunctions}
    \begin{aligned}
      (s_1,\dotsc,s_N) &\mapsto K_{X_1,\dotsc,X_N\,\vert\,X_0=x_0}(s_1,\dotsc,s_N)-\sum_{n=1}^N s_n x_n 
      \\
      \text{and}\qquad (\tau_1,\dotsc,\tau_N) &\mapsto K_{Q_1,\dotsc,Q_N}(\tau_1,\dotsc,\tau_N) - \sum_{n=1}^N \tau_n x_n
    \end{aligned}
  \end{equation}
  are identical under the change of variables $\mat{\tau_1 & \dotsb & \tau_N} = T(s_1,\dotsc,s_N)$.
\end{lemma}

\begin{proof}
By \eqref{KXgivenX0KxiT} and \eqref{TGIdentity},
\begin{align}
  &K_{X_1,\dotsc,X_N\,\vert\,X_0=x_0}(s_1,\dotsc,s_N)-\sum_{n=1}^N s_n x_n 
  \notag\\&\quad
  = K_{\tilde{\xi}_1,\dotsc,\tilde{\xi}_N} \circ T(s_1,\dotsc,s_N) - \mat{s_1 & \dotsb & s_N} \vec{x}
  &&\text{by \eqref{KXgivenX0KxiT}}
  \notag\\&\quad
  = K_{\tilde{\xi}_1,\dotsc,\tilde{\xi}_N} \circ T(s_1,\dotsc,s_N) - \left( \big. T(s_1,\dotsc,s_N) - G\circ T(s_1,\dotsc,s_N) \right) \vec{x}
  &&\text{by \eqref{TGIdentity}}
  \notag\\&\quad
  = \bigl( K_{\tilde{\xi}_1,\dotsc,\tilde{\xi}_N} \circ T(s_1,\dotsc,s_N) + G\circ T(s_1,\dotsc,s_N)\vec{x} \bigr) - T(s_1,\dotsc,s_N) \vec{x}
  \notag\\&\quad
  = K_{Q_1,\dotsc,Q_N}\circ T(s_1,\dotsc,s_N) - T(s_1,\dotsc,s_N) \vec{x}
  &&\text{by \eqref{KQ1QNFormula}.}
  \qedhere
\end{align}
\end{proof}

\begin{proof}[First proof of \refthm{SamplePathProbsFactor}]
The mapping $T$ is smooth, and the identity \eqref{TGIdentity} shows that $T$ has a smooth inverse
\begin{equation}
  T^{-1}(\tau_1,\dotsc,\tau_N) = \mat{\tau_1 & \dotsb & \tau_N} - G(\tau_1, \dotsc, \tau_N)
  .
\end{equation}
In particular, $T$ and $T^{-1}$ have non-singular Jacobians.

Therefore, using \reflemma{FunctionsMappedByT}, the functions $(s_1,\dotsc,s_N) \mapsto K_{X_1,\dotsc,X_N\,\vert\,X_0=x_0}(s_1,\dotsc,s_N)-\sum_{n=1}^N s_n x_n$ and $(\tau_1,\dotsc,\tau_N) \mapsto K_{Q_1,\dotsc,Q_N}(\tau_1,\dotsc,\tau_N) - \sum_{n=1}^N \tau_n x_n$ from \eqref{TwoFunctions} have the same critical points under the change of variables $\mat{\tau_1 & \dotsb & \tau_N} = T(s_1,\dotsc,s_N)$.
These critical points correspond precisely to the saddlepoints $\hat{s}_1^{(X)},\dotsc,\hat{s}_N^{(X)}$ and $\hat{\tau}_1^{(Q)},\dotsc,\hat{\tau}_N^{(Q)}$ in the respective saddlepoint approximations.
Hence, for a given vector $\vec{x}$, whether we can solve the saddlepoint equations, and whether the solution is unique, is the same for both saddlepoint approximations.

The functions from \eqref{TwoFunctions} have Hessian matrices that reduce to $K''_{X_1,\dotsc,X_N\,\vert\,X_0=x_0}(s_1,\dotsc,s_N)$ and $K''_{Q_1,\dotsc,Q_N}(\tau_1,\dotsc,\tau_N)$, respectively.
By \reflemma{FunctionsMappedByT}, at their respective critical points, these Hessian matrices are related in terms of the Jacobian matrix $T'$, namely
\begin{multline}\label{HessianRelation}
  K''_{X_1,\dotsc,X_N\,\vert\,X_0=x_0}\bigl( \hat{s}_1^{(X)},\dotsc,\hat{s}_N^{(X)} \bigr) 
  \\
  = T'\bigl( \hat{s}_1^{(X)},\dotsc,\hat{s}_N^{(X)} \bigr) K''_{Q_1,\dotsc,Q_N}\bigl( \hat{\tau}_1^{(Q)},\dotsc,\hat{\tau}_N^{(Q)} \bigr) T'\bigl( \hat{s}_1^{(X)},\dotsc,\hat{s}_N^{(X)} \bigr)^T
\end{multline}
where ${}^T$ denotes transpose.

For any $\R^D$-valued random variable $R$, the Hessian $K''_R(s)$ is either singular for all $s$ in its domain (if $R$ is supported on a hyperplane of dimension $D-1$ or lower) or positive definite for all $s$ in its domain.
Therefore if one of the Hessians in \eqref{HessianRelation} is singular, then both the Hessians $K''_{X_1,\dotsc,X_N\,\vert\,X_0=x_0}(s_1,\dotsc,s_N)$ and $K''_{Q_1,\dotsc,Q_N}(\tau_1,\dotsc,\tau_N)$ are identically singular, and the corresponding saddlepoint approximations are both undefined (or both equal to $\infty$, depending on our definition in this degenerate case).

Otherwise, the functions in \eqref{TwoFunctions} are both strictly convex and hence yield unique saddlepoints $\hat{s}_1^{(X)},\dotsc,\hat{s}_N^{(X)}$ and $\hat{\tau}_1^{(Q)},\dotsc,\hat{\tau}_N^{(Q)}$.
By \reflemma{FunctionsMappedByT}, the numerators in the corresponding multivariate saddlepoint approximations are equal.
Moreover, by \eqref{HessianRelation} and the squareness of the Jacobian matrix $T'$, the denominators are related by
\begin{multline}
  \sqrt{\det\bigl( 2\pi K''_{X_1,\dotsc,X_N\,\vert\,X_0=x_0}( \hat{s}_1^{(X)},\dotsc,\hat{s}_N^{(X)} ) \bigr)}
  \\
  =
  \sqrt{\det\bigl( 2\pi K''_{Q_1,\dotsc,Q_N}( \hat{\tau}_1^{(Q)},\dotsc,\hat{\tau}_N^{(Q)} ) \bigr) \det\bigl( T'( \hat{s}_1^{(X)},\dotsc,\hat{s}_N^{(X)}) \bigr)^2}
  .
\end{multline}
By \eqref{taunTFormula} and \eqref{tauRecursion}, the block entries of $T$ have the form $\tau_n(s_n,\dotsc,s_N)=s_n+h_n(s_{n+1},\dotsc,s_N)$, where $h_n$ depends only on $s_m$ for $m>n$.
Consequently the Jacobian matrix $T'$ is lower-triangular with ones on the diagonal,
\newlength{\templen} \settowidth{\templen}{$I_{d_N\times d_N}$}
\begin{equation}
  T' = 
  \mat{I_{d_1\times d_1} & 0 & \dotsb & 0 & 0\vphantom{I_{d_N\times d_N}} \\
  \grad_{s_2} h_1 & I_{d_2\times d_2} & \makebox[\templen]{\smash{$\ddots$}} & 0 & 0\vphantom{I_{d_N\times d_N}} \\
  \smash{\vdots} & \smash{\vdots} & \makebox[\templen]{\smash{$\ddots$}} & 0 & 0\vphantom{I_{d_N\times d_N}}\\
  \grad_{s_{N-1}} h_1 & \grad_{s_{N-1}}h_2 & \dotsb & \makebox[\templen]{$I_{d_{N-1}\times d_{N-1}}$} & 0\\
  \grad_{s_N}h_1 & \grad_{s_N}h_2 & \dotsb & \grad_{s_N}h_{N-1} & \makebox[\templen]{$I_{d_N\times d_N}$}
  }
\end{equation}
in block form, where the diagonal blocks $I_{d_n\times d_n}$ denote identity matrices of the indicated sizes.
In particular, $\det(T')=1$, and therefore the denominators of the multivariate saddlepoint approximations $\hat{f}_{X_1,\dotsc,X_N\,\vert\,X_0=x_0}(x_1,\dotsc,x_N)$ and $\hat{f}_{Q_1,\dotsc,Q_N}(x_1,\dotsc,x_N)$ are also equal.
Along with \eqref{prodhatfXgivenhatQ}, this completes the proof.
\end{proof}

\subsection{Second proof of \refthm{SamplePathProbsFactor}}

The second proof interprets the saddlepoint approximation in terms of distributional operations.

\begin{definition}[Exponential tilting]\lbdefn{Tilting}
  Let $X$ be a random variable with values in $\R^d$ and let $s$ be a row vector of dimension $d$ such that $M_X(s)$ is defined.
  A random variable $X^{(s)}$ has the \emph{tilted} distribution (corresponding to $X$ and $s$) if
  \begin{equation}\label{TiltingEhFormula}
    \E(h(X^{(s)}))=\frac{\E(h(X)e^{sX})}{M_X(s)}
  \end{equation}
  for all functions $h\colon\R^d\to\R$ for which the latter expectation is defined.
  Equivalently, if we write $\mu_X$ for the measure defined by $\mu_X(A)=\P(X\in A)$, then the distribution of $X^{(s)}$ is specified by the Radon-Nikodym derivative
  \begin{equation}\label{TiltingRN}
    \frac{d\mu_{X^{(s)}}}{d\mu_X}(x) = \frac{e^{sx}}{M_X(s)} \quad\text{for $\mu_X$-almost-every $x$}
    .
  \end{equation}
\end{definition}
For instance, if $X$ is a continuous random variable with density function $f_X(x)$ with respect to Lebesgue measure on $\R^d$, then $X^{(s)}$ has density function $f_{X^{(s)}}(x) = f_X(x)e^{sx}/M_X(s)$.

Note that \refdefn{Tilting} deals with distributions, rather than the pointwise values of random variables.
Abusing notation slightly, we will ignore the distinction between random variables and their distributions under $\P$, and we will regard $X^{(s)}$ as a generic random variable having the specified distribution.

From \eqref{TiltingEhFormula} we can obtain the MGF and CGF of the tilted distribution:
\begin{equation}\label{TiltedGFs}
  M_{X^{(s)}}(\sigma) = \frac{M_X(s+\sigma)}{M_X(s)}, 
  \qquad 
  K_{X^{(s)}}(\sigma) = K_X(s+\sigma) - K_X(s)
  ,
\end{equation}
In particular, when $s$ is an interior point of the domain of $K_X$, we can differentiate w.r.t.\ $\sigma$ and set $\sigma=0$ to interpret the derivatives $K'(s),K''(s)$ as cumulants of the tilted distribution:
\begin{equation}\label{TiltedCumulants}
  K_X'(s) = \E\left( X^{(s)} \right), \qquad K_X''(s) = \Var\left( X^{(s)} \right)
  ,
\end{equation}
where in the multivariate case $\Var$ means the variance-covariance matrix.
Thus the denominator in the saddlepoint approximation \eqref{UnivariateSPA}/\eqref{MultivariateSPA} can be recognised as the normalising factor for a Gaussian random variable with the same variance-covariance matrix as the tilted distribution.
Moreover, the saddlepoint equation \eqref{UnivariateSE} can be expressed as the condition that the tilted distribution with parameter $\hat{s}=\hat{s}(x)$ should have mean $x$.

By construction, the tilted distribution $X^{(s)}$ is absolutely continuous with respect to the distribution of $X$.
Define the relative entropy, or Kullback-Leibler divergence, by
\begin{equation}\label{HDefinition}
  \HXsX = \int \frac{d\mu_{X^{(s)}}}{d\mu_X}\log\left( \frac{d\mu_{X^{(s)}}}{d\mu_X} \right) d\mu_X
  = \int \log\left( \frac{d\mu_{X^{(s)}}}{d\mu_X} \right) d\mu_{X^{(s)}}
  .
\end{equation}

\begin{lemma}[Tilted relative entropy]\lblemma{TiltedH}
  \begin{equation}\label{TiltedHFormula}
    \HXsX = sK_X'(s) - K_X(s)
    .
  \end{equation}
\end{lemma}
\begin{proof}
By \eqref{TiltingRN} we have $\log\left( \frac{d\mu_{X^{(s)}}}{d\mu_X} \right) = \log\left( \frac{e^{sx}}{M_X(s)} \right) = sx-K_X(s)$.
According to \eqref{HDefinition}, the relative entropy is the expected value of this quantity when $x$ is replaced by $X^{(s)}$:
\begin{align}
  \HXsX
  &=\E\left( sX^{(s)}-K_X(s) \right)
  =s\E\left( X^{(s)} \right)-K_X(s)
\end{align}
which reduces to \eqref{TiltedHFormula} by \eqref{TiltedCumulants}.
\end{proof}

\begin{prop}\lbprop{SPAviaTilting}
  The saddlepoint approximation \eqref{UnivariateSE} and \eqref{UnivariateSPA}/\eqref{MultivariateSPA} can be expressed in terms of tilted distributions by
  \begin{equation}\label{SaddlepointViaTilting}
    \hat{f}_X(x) = \frac{\exp\left( -\bigrelent{X^{(\hat{s})}}{X} \right)}{\sqrt{\det(2\pi \Var(X^{(\hat{s})}))}}
    \quad\text{where $\hat{s}=\hat{s}(x)$ solves}\quad
    \E(X^{(\hat{s})}) = x
    .
  \end{equation}
\end{prop}

\begin{proof}
By construction, the saddlepoint $\hat{s}=\hat{s}(x)$ solves $K_X'(\hat{s})=x$, so $K_X(\hat{s})-\hat{s}x = K_X(\hat{s})-\hat{s}K_X'(\hat{s})$ reduces to a relative entropy as in \reflemma{TiltedH}.
\end{proof}

The next lemma shows that the saddlepoint equation $\E(X^{(\hat{s})}) = x$ uniquely determines the tilted distribution $X^{(\hat{s})}$, in those cases where a solution exists.

\begin{lemma}[Tilted distributions are determined by their means]\lblemma{MeansDetermineTilting}
  Let $X$ be a random variable with values in $\R^d$, and let $s,\tilde{s}$ be $d$-dimensional row vectors for which the tilted distributions $X^{(s)},X^{(\tilde{s})}$ exist and have finite means.
  If $\E(X^{(s)})=\E(X^{(\tilde{s})})$, then $X^{(s)} \equalsd X^{(\tilde{s})}$.
\end{lemma}

\begin{proof}
The hypotheses imply that $K_X(s),K_X(\tilde{s})$ are defined.
From \eqref{TiltedGFs} we find that $\sigma\mapsto K_X(s+\sigma(\tilde{s}-s))-K_X(s)$ is the CGF for the scalar random variable $Y=(\tilde{s}-s)X^{(s)}$, and by convexity this CGF is defined at least for $\sigma\in[0,1]$ and analytic for $\sigma\in(0,1)$.
If $Y$ is non-constant, then $Y^{(\sigma)}$ is also non-constant and has strictly positive variance for all $\sigma\in(0,1)$.
Then $K_Y$ is strictly convex in $(0,1)$ and $K'_Y(\sigma)=(\tilde{s}-s)\E(X^{(s+\sigma(\tilde{s}-s))})$ is strictly increasing in $\sigma\in(0,1)$, and taking $\sigma\decreasesto 0,\sigma\increasesto 1$ shows that $X^{(s)},X^{(\tilde{s})}$ cannot have equal means.
Conversely, if $Y$ is identically constant, then $X^{(s)}$ and hence $X$ are supported in a hyperplane $(\tilde{s}-s)x=c$, and it follows easily that $X^{(s)} \equalsd X^{(\tilde{s})}$.
\end{proof}

The idea of the second proof of \refthm{SamplePathProbsFactor} is to relate the tilted distributions for $X_1,\dotsc,X_N$ and $Q_1,\dotsc,Q_N$, then apply \eqref{SaddlepointViaTilting}.
Key to this strategy is that, for every operation on random variables for which there is a generating function identity, there will be a corresponding distribution identity between their tilted distributions and a corresponding identity between their tilted relative entropies.
The most basic such identities show that tilting preserves independence and additive structure: if $Y_1,Y_2$ are independent and $X=Y_1+Y_2$, then
\begin{equation}\label{TiltingIndependentSum}
  X^{(s)} \equalsd Y_1^{(s)} + Y_2^{(s)}
  \quad\text{and}\quad
  \HXsX = \bigrelent{Y_1^{(s)}}{Y_1} + \bigrelent{Y_2^{(s)}}{Y_2}
\end{equation}
and the tilted joint distribution is given by
    \begin{equation}\label{TiltingIndependentPair}
  \begin{gathered}
    (Y_1,Y_2)^{(s_1, \, s_2)} \equalsd (Y_1^{(s_1)},Y_2^{(s_2)})
    \quad\text{and}\quad
    \\
    \bigrelent{(Y_1,Y_2)^{(s_1, \, s_2)}}{(Y_1,Y_2)} = \bigrelent{Y_1^{(s_1)}}{Y_1} + \bigrelent{Y_2^{(s_2)}}{Y_2}
    .
  \end{gathered}
\end{equation}
Here $Y_1^{(s_1)},Y_2^{(s_2)}$ are taken to be independent, and generally in what follows, different tilted distributions within the same expression are taken to be independent.
In addition, \eqref{TiltingIndependentSum}--\eqref{TiltingIndependentPair} are taken to mean that if the tilted distribution(s) on one side of the equality exist, then so do those on the other side and the equality in distribution holds.
Both \eqref{TiltingIndependentSum}--\eqref{TiltingIndependentPair} can be proved by showing that the two sides have the same CGFs, as in the proof of the more general assertion of \reflemma{TiltingQ} below.

For a process $(Z(t))$ satisfying condition \ref{item:RecCompDisc}, \ref{item:RecCompCts}, or \ref{item:RecCompLin} from \refdefn{RecComp}, \eqref{TiltingIndependentSum} shows that $Z(t+\tau)^{(s)}$ has the same distribution as a sum of independent copies of $Z(t)^{(s)}$ and $Z(\tau)^{(s)}$, for all applicable values $t,\tau$ from $\Z_+$, $\R_+$, or $\R$, respectively.
Thus the tilted distributions $Z(t)^{(s)}$ can be realised as the marginal distributions of a process satisfying the same condition \ref{item:RecCompDisc}, \ref{item:RecCompCts}, or \ref{item:RecCompLin}, and we denote a generic realisation of this process by $(Z^{(s)}(t))$ for $t$ in $\Z_+$, $\R_+$, or $\R$, respectively.

More generally, the following two lemmas show that tilting preserves the class of recursively compounded processes.

\begin{lemma}\lblemma{TiltingQ}
  Let $Q_n(x_0,\dotsc,x_{n-1})$ be as in \refdefn{RecComp}, fix a $d_n$-dimensional row vector $s_n$, and suppose that the tilted distributions $Q_n(x_0,\dotsc,x_{n-1})^{(s_n)}$ exist for all choices of $x_0,\dotsc,x_{n-1}$.
  Then the process $Q_n^{(s_n)}$ defined by
  \begin{equation}
    Q_n^{(s_n)}(x_0,\dotsc,x_{n-1}) = \xi_n^{(s_n)} + \sum_{m=0}^{n-1} \sum_{i=1}^{d_m} Z^{(s_n)}_{n,m,i}(x_{m,i})
  \end{equation}
  satisfies $Q_n^{(s_n)}(x_0,\dotsc,x_{n-1}) \equalsd Q^{(s_n)}(x_0,\dotsc,x_{n-1})$ for all choices of $x_0,\dotsc,x_{n-1}$, where $\xi_n^{(s_n)}$ and the processes $Z^{(s_n)}_{n,m,i}$ are taken to be independent with $Z^{(s_n)}_{n,m,i}(t) \equalsd Z_{n,m,i}(t)^{(s_n)}$, for all applicable values of $t$ as in \refdefn{RecComp}.
  Furthermore
  \begin{multline}\label{TiltedQRelEnt}
    \bigrelent{Q^{(s_n)}_n(x_0,\dotsc,x_{n-1})}{Q_n(x_0,\dotsc,x_{n-1})}
    = \bigrelent{\xi_n^{(s_n)}}{\xi_n} 
    + \sum_{m=0}^{n-1} \left[ s_n g'_{n,m}(s_n) - g_{n,m}(s_n) \right] x_m
  \end{multline}
  and the analogues for $Q^{(s_n)}_n$ of the row-vector-valued functions $g_{n,m}$ given by \eqref{gnmiFormula} are
  \begin{equation}\label{Tiltedg}
    \tilde{g}_{n,m}(\sigma_n) = g_{n,m}(s_n+\sigma_n) - g_{n,m}(s_n)
    .
  \end{equation}
\end{lemma}

\begin{proof}
The existence of the tilted distribution $Q_n(x_0,\dotsc,x_{n-1})^{(s_n)}$ is equivalent to the existence of $K_{Q_n(x_0,\dotsc,x_{n-1})}(s_n)$ as a finite real number.
By \eqref{KQnFormula}, it follows that $K_{\xi_n}(s_n)$ and $K_{Z_{n,m,i}(1)}(s_n)$ exist for all $m$ and $i$.
Since $K_{Z_{n,m,i}(t)}(s_n) = t K_{Z_{n,m,i}(1)}(s_n)$ for all applicable values of $t$ from \refdefn{RecComp}, it follows that all tilted distributions $\xi_n^{(s_n)},Z_{n,m,i}(x_{m,i})^{(s_n)}$ exist for all applicable choices of $x_0,\dotsc,x_n$.
As explained above, the tilted distributions $Z_{n,m,i}(t)^{(s_n)}$ may be taken as the marginal distributions of processes $Z^{(s_n)}_{n,m,i}(t)$, and thus $Q_n^{(s_n)}$ matches the construction of \refdefn{RecComp}.
Equation \eqref{Tiltedg} follows as the vector analogue of \eqref{TiltedGFs} since $\tilde{g}_{n,m}$ has entries
\begin{equation}
  \tilde{g}_{n,m,i}(\sigma_n) = K_{Z_{n,m,i}(1)^{(s_n)}}(\sigma_n) = K_{Z_{n,m,i}(1)}(s_n+\sigma_n)-K_{Z_{n,m,i}(1)}(s_n) = g_{n,m,i}(s_n+\sigma_n)-g_{n,m,i}(s_n).
\end{equation}
By \eqref{KQnFormula} and \eqref{TiltedGFs},
\begin{align}
  K_{Q_n(x_0,\dotsc,x_{n-1})^{(s_n)}}(\sigma_n) 
  &= K_{Q_n(x_0,\dotsc,x_{n-1})}(s_n+\sigma_n) - K_{Q_n(x_0,\dotsc,x_{n-1})}(s_n)
  \notag\\&
  = K_{\xi_n}(s_n+\sigma_n) - K_{\xi_n}(s_n) + \sum_{m=0}^{n-1}  [g_{n,m}(s_n+\sigma_n) - g_{n,m}(s_n)] x_m
  \notag\\&
  = K_{\xi_n^{(s_n)}}(\sigma_n) + \sum_{m=0}^{n-1} \tilde{g}_{n,m}(\sigma_n) x_m
  ,
\end{align}
which is the same as applying \eqref{KQnFormula} and \eqref{Tiltedg} to $Q^{(s_n)}_n(x_0,\dotsc,x_{n-1})$.
Hence $Q_n(x_0,\dotsc,x_{n-1})^{(s_n)}$ and $Q^{(s_n)}_n(x_0,\dotsc,x_{n-1})$ have the same CGF and hence the same distribution, as claimed.
Then \reflemma{TiltedH} and \eqref{KQnFormula} give
\begin{align}
  &\bigrelent{Q_n^{(s_n)}(x_0,\dotsc,x_{n-1})}{Q_n(x_0,\dotsc,x_{n-1})}
  \notag\\&\quad
  = s_n K'_{\xi_n}(s_n) - K_{\xi_n}(s_n) + \sum_{m=0}^{n-1} \left[ s_n g'_{n,m}(s_n) - g_{n,m}(s_n) \right] x_m
\end{align}
which reduces to \eqref{TiltedQRelEnt} by \reflemma{TiltedH} applied to $\xi_n$.
\end{proof}

Note that for fixed $s_n$ for which $Q^{(s_n)}_{n}(x_0,\dotsc,x_{n-1})$ exists, the right-hand side of \eqref{TiltedQRelEnt} is well-defined as a function of arbitrary $x_0,\dotsc,x_{n-1}$.
Abusing notation slightly, we will use $\bigrelent{Q^{(s_n)}_{n}(x_0,\dotsc,x_{n-1})}{Q_n(x_0,\dotsc,x_{n-1})}$ as a shorthand for the right-hand side of \eqref{TiltedQRelEnt}, even if requirements from \refdefn{RecComp}\ref{item:RecCompDisc} (for certain entries of $x_0,\dotsc,x_{n-1}$ to be integers) are violated.
Similarly, we will write
\begin{equation}\label{EVarTiltedQ}
  \begin{gathered}
    \E\left( Q^{(s_n)}_{n}(x_0,\dotsc,x_{n-1}) \right) = \E \left( \xi_n^{(s_n)} \right) + \sum_{m=0}^{n-1}\sum_{i=1}^{d_m} \E\left( Z_{n,m,i}(1)^{(s_n)} \right) x_{m,i}
    ,
    \\
    \Var\left( Q^{(s_n)}_{n}(x_0,\dotsc,x_{n-1}) \right) = \Var \left( \xi_n^{(s_n)} \right) + \sum_{m=0}^{n-1}\sum_{i=1}^{d_m} \Var\left( Z_{n,m,i}(1)^{(s_n)} \right) x_{m,i}
  \end{gathered}
\end{equation}
for all choices of $x_0,\dotsc,x_{n-1}$, with the same convention.

\begin{lemma}\lblemma{PreserveRecComp}
  If $(X_0,X_1,\dotsc,X_N)$ is a recursively compounded process, then the tilted joint distribution, which we abbreviate as $(\tilde{X}_0,\tilde{X}_1,\dotsc,\tilde{X}_N) \equalsd (X_0,\dotsc,X_N)^{(s_0,\dotsc,s_N)}$, is also a recursively compounded process.
  We may construct this process via
  \begin{equation}\label{TiltedRCPConstruction}
    \tilde{X}_0 \equalsd X_0^{(\tau_0(s_0,\dotsc,s_N))} 
    \quad\text{and}\quad
    \tilde{X}_n = Q_n^{(\tau_n(s_n,\dotsc,s_N))}(\tilde{X}_0,\dotsc,\tilde{X}_{n-1}) \quad\text{for }n=1,\dotsc,N
    ,
  \end{equation}
  where each tilted process $Q_n^{(\tau_n(s_n,\dotsc,s_N))}(x_0,\dotsc,x_{n-1})$ (over all choices of $x_0,\dotsc,x_{n-1}$) is chosen as in \reflemma{TiltingQ}, independently of $\tilde{X}_0,\dotsc,\tilde{X}_{n-1}$.
  Moreover
  \begin{multline}
    \label{TiltedRCPRelEnt}
    \bigrelent{(\tilde{X}_0,\dotsc,\tilde{X}_N)}{(X_0,\dotsc,X_N)} 
    \\
    = \bigrelent{X_0^{(\tau_0(s_0,\dotsc,s_N))}}{X_0} 
    + \sum_{n=1}^N \bigrelent{Q_n^{(\tau_n(s_n,\dotsc,s_N))}(\tilde{x}_0,\dotsc,\tilde{x}_{n-1})}{Q_n(\tilde{x}_0,\dotsc,\tilde{x}_{n-1})}
    \\
    \text{evaluated with }\tilde{x}_0=\E(\tilde{X}_0), \dotsc,\tilde{x}_{N-1}=\E(\tilde{X}_{N-1})
    .
  \end{multline}
\end{lemma}

\begin{proof}
We begin with the case $N=1$.
Fix $(s_0,s_1)$.
To show the equality in distribution, it suffices to show that the tilted multivariate CGF given by \eqref{TiltedGFs} coincides with the CGF of the recursively compounded process $(\tilde{X}_0,\tilde{X}_1)$ defined by \eqref{TiltedRCPConstruction}.
The latter CGF, $K_{\tilde{X}_0,\tilde{X}_1}(\sigma_0,\sigma_1)$, can be computed by replacing the functions $K_{X_0}$, $K_{\xi_1}$, and $g_{1,0}$ from \reflemma{KX0XN} by the functions
\begin{equation}\label{TiltedX0xi1g10}
  \begin{aligned}
    K_{\tilde{X}_0}(\sigma_0) &= K_{X_0}(s_0+g_{1,0}(s_1)+\sigma_0)-K_{X_0}(s_0+g_{1,0}(s_1))
    ,
    \\
    K_{\tilde{\xi_1}}(\sigma_1) &= K_{\xi_1}(s_1+\sigma_1) - K_{\xi_1}(s_1)
    ,
    \\
    \tilde{g}_{1,0}(\sigma_1) &= g_{1,0}(s_1+\sigma_1) - g_{1,0}(s_1)
  \end{aligned}
\end{equation}
computed using \eqref{TiltedGFs} and \reflemma{TiltingQ}.
The functions $\tau_0,\tau_1$ from \reflemma{KX0XN} become
\begin{equation}
  \tilde{\tau}_1(\sigma_1) = \sigma_1
  , \quad
  \tilde{\tau}_0(\sigma_0,\sigma_1) = \sigma_0 + \tilde{g}_{1,0}(\sigma_1)
  .
\end{equation} 

On the other hand, applying \eqref{TiltedGFs} and \reflemma{KX0XN} to $(X_0,X_1)$, we compute
\begin{align}
  K_{(X_0,X_1)^{(s_0,s_1)}}(\sigma_0,\sigma_1) 
  &= 
  K_{X_0,X_1}(s_0+\sigma_0,s_1+\sigma_1) - K_{X_0,X_1}(s_0,s_1)
  \notag\\&
  = K_{X_0}(\tau_0(s_0+\sigma_0,s_1+\sigma_1)) - K_{X_0}(\tau_0(s_0,s_1)) 
  \notag\\&\quad
  + K_{\xi_1}(s_1+\sigma_1) - K_{\xi_1}(s_1)
  \notag\\&
  = K_{X_0}(s_0+\sigma_0+g_{1,0}(s_1+\sigma_1)) - K_{X_0}(s_0+g_{1,0}(s_1)) + K_{\tilde{\xi}_1}(\sigma_1)
  .
\end{align}
Substituting $g_{1,0}(s_1+\sigma_1)=g_{1,0}(s_1)+\tilde{g}_{1,0}(\sigma_1)$, this yields
\begin{align}
  K_{(X_0,X_1)^{(s_0,s_1)}}(\sigma_0,\sigma_1) 
  &= 
  K_{X_0}(s_0+\sigma_0+g_{1,0}(s_1)+\tilde{g}_{1,0}(\sigma_1)) - K_{X_0}(s_0+g_{1,0}(s_1)) + K_{\tilde{\xi}_1}(\sigma_1)
  \notag\\&
  = K_{\tilde{X}_0}(\sigma_0+\tilde{g}_{1,0}(\sigma_1)) + K_{\tilde{\xi}_1}(\sigma_1)
  \notag\\&
  = K_{\tilde{X}_0}(\tilde{\tau}_0(\sigma_0,\sigma_1)) + K_{\tilde{\xi}_1}(\tilde{\tau}_1(\sigma_1))
  ,
\end{align}
and by \reflemma{KX0XN} this matches with the CGF of the recursively compounded process $(\tilde{X}_0,\tilde{X}_1)$ given by \eqref{TiltedRCPConstruction}.
Hence $(X_0,X_1)^{(s_0,s_1)} \equalsd (\tilde{X}_0,\tilde{X}_1)$, as required.

For the tilted relative entropy, note from $K_{X_0,X_1}(s_0,s_1) = K_{X_0}(s_0 + g_{1,0}(s_1)) + K_{\xi_1}(s_1)$ that 
\begin{equation}
  \begin{aligned}
    \grad_{s_0}K_{X_0,X_1}(s_0,s_1) &= K_{X_0}'(s_0 + g_{1,0}(s_1))
    ,
    \\
    \grad_{s_1}K_{X_0,X_1}(s_0,s_1) &= g_{1,0}'(s_1)K_{X_0}'(s_0 + g_{1,0}(s_1))+K_{\xi_1}'(s_1)
    ,
  \end{aligned}
\end{equation}
and write 
\begin{equation}
  \tilde{x}_0=\E(\tilde{X}_0) = \E(X_0^{(s_0+g_{1,0}(s_1))}) =K'_{X_0}(s_0+g_{1,0}(s_1))
  , 
\end{equation}
where the final equality uses \eqref{TiltedCumulants}.
Then
\begin{align}
  &\bigrelent{(X_0,X_1)^{(s_0,s_1)}}{(X_0,X_1)} 
  \notag\\&\quad
  = s_0 \grad_{s_0}K_{X_0,X_1}(s_0,s_1) + s_1 \grad_{s_1}K_{X_0,X_1}(s_0,s_1) - K_{X_0,X_1}(s_0,s_1)
  \notag\\&\quad
  = s_0 K_{X_0}'(s_0 + g_{1,0}(s_1)) + s_1 g_{1,0}'(s_1)K_{X_0}'(s_0 + g_{1,0}(s_1))
  \notag\\&\qquad 
  +s_1 K_{\xi_1}'(s_1)- K_{X_0}(s_0 + g_{1,0}(s_1)) - K_{\xi_1}(s_1)
  \notag\\&\quad
  = \left( \Big. [s_0 + g_{1,0}(s_1)] K_{X_0}'(s_0 + g_{1,0}(s_1)) - g_{1,0}(s_1) \tilde{x}_0 \right) + s_1 g_{1,0}'(s_1)\tilde{x}_0
  \notag\\&\qquad 
   + s_1 K_{\xi_1}'(s_1) - K_{\xi_1}(s_1) - K_{X_0}(s_0 + g_{1,0}(s_1))
  \notag\\&\quad 
  = \bigrelent{X_0^{(s_0 + g_{1,0}(s_1))}}{X_0} + \bigrelent{\xi_1^{(s_1)}}{\xi_1} + [s_1 g_{1,0}'(s_1) - g_{1,0}(s_1)] \tilde{x}_0
  \notag\\&\quad
  = \bigrelent{X_0^{(s_0 + g_{1,0}(s_1))}}{X_0} + \bigrelent{Q_1^{(s_1)}(\tilde{x}_0)}{Q_1(\tilde{x}_0)}
  \label{TiltedRCPHN=1}
\end{align}
by \eqref{TiltedQRelEnt}. 
This reduces to \eqref{TiltedRCPRelEnt}, thus completing the proof of \reflemma{PreserveRecComp} for $N=1$.

For the remainder of the proof, fix $N>1$ and tilting parameters $(s_0,\dotsc,s_N)$.
We will proceed by induction in reverse time order, at each stage replacing the last tilted variable by one additional step of a (tilted) recursively compounded process.
Specifically, we will show by induction that for $n=N,N-1,\dotsc,1,0$, 
\begin{multline}\label{TiltingInduction}
  (X_0,\dotsc,X_N)^{(s_0,\dotsc,s_N)} \equalsd (\tilde{X}_{0;n},\dotsc,\tilde{X}_{N;n})
  \\
  \text{where }(\tilde{X}_{0;n}, \dotsc,\tilde{X}_{n;n}) \equalsd (X_0,\dotsc,X_n)^{(\rho_{n,0},\dotsc,\rho_{n,n})}
  \\
  \text{and }\tilde{X}_{k;n} = Q_k^{(\tau_k(s_k,\dotsc,s_N))}(\tilde{X}_{0,n},\dotsc,\tilde{X}_{k-1,n})\text{ for all }k>n
  ,
\end{multline}
where throughout \eqref{TiltingInduction} we assume that all realisations of the processes $Q_k^{(\tau_k(s_k,\dotsc,s_N))}$ are chosen independently of $\tilde{X}_{0;n},\tilde{X}_{k;n}$, and where $\rho_{n,m}=\rho_{n,m}(s_m,\dotsc,s_N)$ are the quantities from \eqref{rhonmFormula}.
Recalling that $\rho_{N,m}=s_m$ for all $m$, we see that \eqref{TiltingInduction} holds trivially in the base case $n=N$.

To advance the induction, suppose that $1<n\leq N$ and that \eqref{TiltingInduction} holds for $n$.
Note that a recursively compounded process can be ``collapsed'' to one of smaller length: if we define $\collapsedrcp{X}_0=(X_0,\dotsc,X_{n-1})$, $\collapsedrcp{X}_1=X_n$, then $(\collapsedrcp{X}_0,\collapsedrcp{X}_1)$ is itself a recursively compounded process, say $\collapsedrcp{X}_1=\collapsedrcp{Q}_1(\collapsedrcp{X}_0)$.
Using the notation from \refdefn{RecComp}, we have $\collapsedrcp{\xi}_1=\xi_n$, $\collapsedrcp{d}_0=d_0+\dotsb+d_{n-1}$, $\collapsedrcp{d}_1=d_n$, and the processes $(\collapsedrcp{Z}_{1,0,i})_{i=1,\dotsc,\collapsedrcp{d}_0}$ are the processes $(Z_{n,m,j})_{m=0,\dotsc,n-1, \; j=1,\dotsc,d_m}$ arranged in lexicographic order.
With these conventions, $\collapsedrcp{Q}_1(\collapsedrcp{X}_0)=Q_n(X_0,\dotsc,X_{n-1})$, and the function $\collapsedrcp{g}_{1,0}(\collapsedrcp{\sigma}_1)$ from applying \eqref{gnmiFormula} to $(\collapsedrcp{X}_0,\collapsedrcp{X}_1)$ is formed by concatenating the blocks $g_{n,m}(\sigma_n)$ for $m=0,\dotsc,n-1$.

Multivariate CGFs and tilted distributions are not affected by how we group vectors into blocks.
Therefore, if we write the corresponding multivariate generating function arguments in block form as $\collapsedrcp{\sigma}_0=\mat{\sigma_0 & \dotsb & \sigma_{n-1}}$ and $\collapsedrcp{\sigma}_1=\sigma_n$, where each $\sigma_m$ is a row vector of dimension $d_m$, then $(X_0,\dotsc,X_n)^{(\sigma_0,\dotsc,\sigma_n)} \equalsd (\collapsedrcp{X}_0,\collapsedrcp{X}_1)^{(\collapsedrcp{\sigma}_0,\collapsedrcp{\sigma}_1)}$.

By the results already proved for $N=1$,
\begin{align}
  (X_0,\dotsc,X_n)^{(\sigma_0,\dotsc,\sigma_n)} 
  &\equalsd 
  (\collapsedrcp{X}_0,\collapsedrcp{X}_1)^{(\collapsedrcp{\sigma}_0,\collapsedrcp{\sigma}_1)}
  \equalsd 
  (\collapsedrcp{Y}_0,\collapsedrcp{Y}_1)
  \notag\\&\qquad
  \text{where }
  \collapsedrcp{Y}_0 \equalsd \collapsedrcp{X}_0^{(\collapsedrcp{\sigma}_0+\collapsedrcp{g}_{1,0}(\collapsedrcp{\sigma}_1))}
  \text{ and }
  \collapsedrcp{Y}_1 \equalsd \collapsedrcp{Q}_1^{(\collapsedrcp{\sigma}_1)}(\collapsedrcp{Y}_0)
  .
\end{align}
Write $\collapsedrcp{Y}_1=Y_n,\collapsedrcp{Y}_0=(Y_0,\dotsc,Y_{n-1})$ in block form.
Then, by the block structure above, 
\begin{equation}\label{YTiltedBlockForm}
  (Y_0,\dotsc,Y_{n-1}) \equalsd (X_0,\dotsc,X_{n-1})^{(\sigma_0 + g_{n,0}(\sigma_n), \dotsc, \sigma_{n-1} + g_{n,n-1}(\sigma_n))} 
  ,\quad
  Y_n \equalsd Q_n^{(\sigma_n)}(Y_0,\dotsc,Y_{n-1})
  .
\end{equation}

Henceforth set $\sigma_m=\rho_{n,m}$ for $m=0,\dotsc,n$.
By \eqref{rhonntaun}--\eqref{rhonmRecursion}, the tilting parameters $\sigma_m+g_{n,m}(\sigma_n)$ from \eqref{YTiltedBlockForm} reduce to $\rho_{n-1,m}$, and moreover $\sigma_n=\rho_{n,n}=\tau_n(s_n,\dotsc,s_N)$.
Thus the construction \eqref{YTiltedBlockForm} of $(Y_0,\dotsc,Y_{n-1},Y_n)$ is equivalent to the construction from \eqref{TiltingInduction} (with $n$ replaced by $n-1$) of $(\tilde{X}_{0;n-1},\dotsc,\tilde{X}_{n-1;n-1},\tilde{X}_{n;n-1})$.
Therefore
\begin{equation}
  (\tilde{X}_{0;n},\dotsc,\tilde{X}_{n;n}) \equalsd (X_0,\dotsc,X_N)^{(s_0,\dotsc,s_N)} \equalsd (Y_0,\dotsc,Y_n) \equalsd (\tilde{X}_{0;n-1},\dotsc,\tilde{X}_{n;n-1})
  .
\end{equation}
That is, the processes $(\tilde{X}_{0;n},\dotsc,\tilde{X}_{N;n})$ and $(\tilde{X}_{0;n-1},\dotsc,\tilde{X}_{N;n-1})$ have the same joint distribution up to time $n$.
But for both processes, their distribution after time $n$ is determined by their joint distribution up to time $n$ and the same independent randomness via the $Q_k^{(\tau_k(s_k,\dotsc,s_N))}$'s for $k>n$.
Hence the two processes have the same joint distribution and
\begin{equation}
  (\tilde{X}_{0;n-1},\dotsc,\tilde{X}_{N;n-1}) \equalsd (\tilde{X}_{0;n},\dotsc,\tilde{X}_{N;n}) \equalsd (X_0,\dotsc,X_N)^{(s_0,\dotsc,s_N)}
  .
\end{equation}
This advances the induction.
Hence \eqref{TiltingInduction} holds for all $n=N,\dotsc,0$.
Taking $n=0$ verifies that the process $(\tilde{X}_0,\dotsc,\tilde{X}_N)$ defined in \eqref{TiltedRCPConstruction} matches the tilted distribution, as claimed.

A similar argument gives \eqref{TiltedRCPRelEnt}.
Maintaining previous notation, we apply \eqref{TiltedRCPHN=1} to $(\collapsedrcp{X}_0,\collapsedrcp{X}_1)$ and its tilting $(\collapsedrcp{Y}_0,\collapsedrcp{Y}_1)$, noting that $\tilde{x}_0$ from \eqref{TiltedRCPHN=1} is replaced by $\E(\collapsedrcp{Y}_0)=(\tilde{x}_0,\dotsc,\tilde{x}_{n-1})$.
Using \eqref{TiltingInduction} twice,
\begin{align}
  &\bigrelent{(\tilde{X}_0,\dotsc,\tilde{X}_n)}{(X_0,\dotsc,X_n)} 
  \\&\quad
  = \bigrelent{(X_0,\dotsc,X_n)^{(\rho_{n,0},\dotsc,\rho_{n,n})}}{(X_0,\dotsc,X_n)} 
  \notag\\&\quad
  = \bigrelent{(\collapsedrcp{Y}_0,\collapsedrcp{Y}_1)}{(\collapsedrcp{X}_0,\collapsedrcp{X}_1)}
  \notag\\&\quad
  = \bigrelent{\collapsedrcp{Y}_0}{\collapsedrcp{X}_0} + \bigrelent{\collapsedrcp{Q}_1^{(\rho_{n,n})}(\E(\collapsedrcp{Y}_0))}{\collapsedrcp{Q}_1(\E(\collapsedrcp{Y}_0))}
  \notag\\&\quad
  = \bigrelent{(\tilde{X}_0,\dotsc,\tilde{X}_{n-1})}{(X_0,\dotsc,X_n)} + \bigrelent{Q_n^{(\tau_n(s_n,\dotsc,s_N))}(\tilde{x}_0,\dotsc,\tilde{x}_{n-1})}{Q_n(\tilde{x}_0,\dotsc,\tilde{x}_{n-1})}
  \notag
  ,
\end{align}
and \eqref{TiltedRCPRelEnt} follows by induction.
\end{proof}

\begin{proof}[Proof of \refthm{SamplePathProbsFactor}]
For $N=1$ there is nothing to prove.
Similar to the argument in the previous proof, it suffices to consider the case $N=2$, since for $N>2$ we can set $\collapsedrcp{X}_1=(X_1,\dotsc,X_{N-1})$ and $\collapsedrcp{X}_2=X_N$ and apply an inductive argument.
Moreover, without loss of generality we may assume that $x_0=0$; if not, replace each $\xi_n$ by $\tilde{\xi}_n$ from \eqref{tildexinFormula}, as in the earlier proof of \refthm{SamplePathProbsFactor}.

Fix $x_1,x_2$.
Then we wish to compare the joint distributions of $(X_1,X_2)$ versus $(Y_1,Y_2)$, where we define
\begin{equation}
  X_1=Q_1(0)=\xi_1 = Y_1 
  \qquad\text{and}\qquad
  \begin{aligned}
    X_2 &= Q_2(0,X_1)
    ,
    \\
    Y_2 &= Q_2(0,x_1)
    .
  \end{aligned}
\end{equation}
Note that $Y_1,Y_2$ are independent.

With tilting parameters $(s_1,s_2)$ and $(\sigma_1,\sigma_2)$ to be chosen later, write 
\begin{equation}
  (\hat{X}_1,\hat{X}_2)\equalsd (X_1,X_2)^{(s_1,s_2)} 
  ,
  \qquad
  (\check{Y}_1,\check{Y}_2)\equalsd (Y_1,Y_2)^{(\sigma_1,\sigma_2)} 
\end{equation}
for the corresponding tilted joint distributions.
As in \eqref{TiltingIndependentPair}, the tilted random variables $\check{Y}_1,\check{Y}_2$ are independent and can be taken to satisfy 
\begin{equation}
  \check{Y}_1 = \check{Q}_1(0) \equalsd Q_1^{(\sigma_1)}(0), 
  \qquad 
  \check{Y}_2 = \check{Q}_2(0,x_1) \equalsd Q_2^{(\sigma_2)}(0,x_1)
  ,
\end{equation}
independently.
The tilting parameters $\sigma_1,\dotsc,\sigma_N$ can be varied separately, so the processes $\check{Q}_1,\check{Q}_2$, considered jointly, can realise any desired tiltings of the processes $Q_1,Q_2$.

On the other hand, by \reflemma{PreserveRecComp}, the tilted joint distribution $(\hat{X}_1,\hat{X}_2)$ is itself a recursively compounded process and can be taken to satisfy
\begin{equation}
  \hat{X}_1 = \hat{Q}_1(0) \equalsd Q_1^{(\tau_1(s_1,s_2))}(0),
  \qquad
  \hat{X}_2 = \hat{Q}_2(0,\hat{X}_1)
\end{equation}
where the process $a_1\mapsto \hat{Q}_2(0,a_1)$ is a tilted process $\hat{Q}_2 \equalsd Q_2^{(\tau_2(s_2))}$ independent of $X_1$.
From \reflemma{KX0XN}, we see that here too the processes $\hat{Q}_1,\hat{Q}_2$, considered jointly, can realise any desired tiltings of the processes $Q_1,Q_2$: since $\tau_2(s_2)=s_2$, we can obtain the desired tilting of $Q_2$ by varying $s_2$; then we can obtain any desired tilting parameter $\tau_1(s_1,s_2) = s_1+g_{2,1}(s_2)$ by varying $s_1$ with $s_2$ held fixed.

From \eqref{EVarTiltedQ} we see that $\condE{ \! \hat{X}_2}{\hat{X}_1}$ and $\Var\condparentheses{ \! \hat{X}_2}{\hat{X}_1}$ depend affinely on $\hat{X}_1$: writing $\hat{\xi}_2$ and $(\hat{Z}_{2,1,i})_{i=1,\dotsc,d_1}$ for the random variables from \refdefn{RecComp} for $\hat{Q}_2$,
\begin{equation}
  \begin{aligned}
    \condE{ \! \hat{X}_2}{\hat{X}_1} = \E(\hat{\xi}_2) + \textstyle{\sum_i} \E(\hat{Z}_{2,1,i}(1)) \hat{X}_{1,i} = \E(\hat{Q}_2(0,\hat{x}_1))&
    \\
    \Var\condparentheses{ \! \hat{X}_2}{\hat{X}_1} = \Var(\hat{\xi}_2) + \textstyle{\sum_i} X_{1,i}\Var(\hat{Z}_{2,1,i}(1)) = \Var(\hat{Q}_2(0,\hat{x}_1))&
  \end{aligned}
  \quad\text{evaluated at }\hat{x}_1=\hat{X}_1
  .
\end{equation}
In particular,
\begin{equation}\label{EEcondVarhatX2}
  \begin{aligned}
    \E\bigl( \hat{X}_2 \bigr) &= \E(\hat{Q}_2(0,\E(\hat{X}_1)))
    ,
    \\
    \E\bigl( \Var\condparentheses{ \! \hat{X}_2}{\hat{X}_1} \bigr) &= \Var(\hat{Q}_2(0,\E(\hat{X}_1)))
    .
  \end{aligned}
\end{equation}
Write $\hat{A}$ for the $d_2\times d_1$ matrix whose $i^\text{th}$ column is $\E(\hat{Z}_{2,1,i})$.
Then 
\begin{equation}\label{condEhatX2hatA}
  \condE{ \! \hat{X}_2}{\hat{X}_1} = \E(\hat{\xi}_2)+\hat{A}\hat{X}_1
\end{equation} 
and the Variance Partition Formula yields the variance-covariance matrix
\begin{align}
  \Var\mat{\hat{X}_1 \\ \hat{X}_2} &= \mat{\Var(\hat{X}_1) & \Cov(\hat{X}_1, \hat{X}_2) \\ \Cov(\hat{X}_2,\hat{X}_1) & \Var(\hat{X}_2)}
  \notag\\&
  = \mat{\Var(\hat{X}_1) & \Cov\left( \hat{X}_1, \condE{ \! \hat{X}_2}{\hat{X}_1} \right) \\ \Cov\left( \condE{ \! \hat{X}_2}{\hat{X}_1},\hat{X}_1 \right) & \Var\left( \condE{ \! \hat{X}_2}{\hat{X}_1} \right)+\E\left( \Var\condparentheses{ \! \hat{X}_2}{\hat{X}_1} \right)}
  \notag\\&
  = \mat{\Var(\hat{X}_1) & \Var(\hat{X}_1)\hat{A}^T \\ \hat{A}\Var(\hat{X}_1) & \hat{A}\Var(\hat{X}_1)\hat{A}^T + \Var(\hat{Q}_2(0,\E(\hat{X}_1))) }
  \label{VarX1X2hat}
\end{align}
by \eqref{EEcondVarhatX2}--\eqref{condEhatX2hatA} and $\Cov(AX,BY)=A\Cov(X,Y)B^T$ for deterministic matrices $A,B$.

According to \refprop{SPAviaTilting}, the saddlepoint equations can be expressed in terms of the tilted means $\E(\hat{X}_1),\E(\check{Y}_1)$ and $\E(\hat{X}_2),\E(\hat{Y}_2)$.
Since $X_1=Y_1$, both $\hat{X}_1$ and $\hat{Y}_1$ are tiltings of the same random variable, and as remarked above both tilting parameters can have arbitrary values.
So the set of attainable tilted means (by varying the respective tilting parameters) is the same for both distributions, and thus $\E(\hat{X}_1)=x_1$ has a solution if and only if $\E(\check{Y}_1)=x_1$ does.
Moreover, if $\E(\hat{X}_1)=x_1$ holds, then \eqref{EEcondVarhatX2} shows that $\E(\hat{X}_2),\E(\hat{Y}_2)$ reduce to $\E(\hat{Q}_2(0,x_1)),\E(\check{Q}_2(0,x_1))$.
The latter two expectations are again arbitrary tiltings of the same random variable, yielding the same sets of attainable tilted means.
We conclude that there exists a solution to the saddlepoint equations for the random pair $(X_1,X_2)$ and the data values $(x_1,x_2)$,
\begin{equation}\label{SEhatX12}
  \E(\hat{X}_1)=x_1
  ,
  \quad
  \E(\hat{X}_2)=x_2
  ,
\end{equation}
if and only if there exists a solution to
\begin{equation}\label{SEcheckY12}
  \E(\check{Y}_1)=x_1
  ,
  \quad
  \E(\check{Y}_2)=x_2
  ,
\end{equation}
the saddlepoint equations for the random pair $(Y_1,Y_2)$ with the same data values.

Now suppose that the tilting parameters $(s_1,s_2)$ and $(\sigma_1,\sigma_2)$ are chosen so that  \eqref{SEhatX12}--\eqref{SEcheckY12} hold.
Since $X_1=Y_1$, we conclude that $\hat{X}_1$ and $\check{Y}_1$ are tiltings of the same distribution with the same tilted means.
By \reflemma{MeansDetermineTilting}, $\hat{X}_1 \equalsd \check{Y}_1$.
Moreover by \eqref{EEcondVarhatX2}, 
\begin{equation}
  \E(\check{Q}_2(0,x_1)) = \E(\check{Y}_2) = \E(\hat{X}_2) = \E(\hat{Q}_2(0,\E(\hat{X}_1))) = \E(\hat{Q}_2(0,x_1))
  ,
\end{equation}
so another application of \reflemma{MeansDetermineTilting} gives $\hat{Q}_2(0,x_1) \equalsd \check{Q}_2(0,x_1)$.
From \reflemma{PreserveRecComp} it follows that $\bigrelent{(\hat{X}_1,\hat{X}_2)}{(X_1,X_2)}=\bigrelent{(\check{Y}_1,\check{Y}_2)}{(Y_1,Y_2)}$.
The two variance-covariance matrices are not equal in general, but from \eqref{VarX1X2hat} and \eqref{SEhatX12} we have
\begin{align}
  \Var\mat{\hat{X}_1 \\ \hat{X}_2} 
  &= 
  \mat{I & 0 \\ A & I} \mat{\Var(\hat{X}_1) & 0 \\ 0 & \Var(\hat{Q}_2(0,x_1))} \mat{I & A^T \\ 0 & I}
  \notag\\&
  = \mat{I & 0 \\ A & I} \mat{\Var(\check{Y}_1) & 0 \\ 0 & \Var(\check{Q}_2(0,x_1))} \mat{I & A^T \\ 0 & I}
  \notag\\&
  = \mat{I & 0 \\ A & I} \Var\mat{\check{Y}_1 \\ \check{Y}_2} \mat{I & A^T \\ 0 & I}
\end{align}
and therefore $\Var\mat{\hat{X}_1 \\ \hat{X}_2}, \Var\mat{\check{Y}_1 \\ \check{Y}_2}$ have equal determinants.
By \refprop{SPAviaTilting}, 
\begin{align*}
  \hat{f}_{X_1,X_2}(x_1,x_2) 
  &= 
  \frac{\exp\left( -\bigrelent{(\hat{X}_1,\hat{X}_2)}{(X_1,X_2)} \right)}{\sqrt{\det\left( 2\pi \Var\smallmat{\hat{X}_1 \\ \hat{X}_2} \right)}} 
  \\&
  = \frac{\exp\left( -\bigrelent{(\check{Y}_1,\check{Y}_2)}{(Y_1,Y_2)} \right)}{\sqrt{\det\left( 2\pi \Var\smallmat{\check{Y}_1 \\ \check{Y}_2} \right)}} 
  = \hat{f}_{Y_1,Y_2}(x_1,x_2)
  .
  \qedhere
\end{align*}
\end{proof}


\begin{thebibliography}{1}
  
  \bibitem{Butler2007}
  Ronald~W. Butler.
  \newblock {\em Saddlepoint approximations with applications}.
  \newblock Cambridge series on statistical and probabilistic mathematics.
  Cambridge University Press, Cambridge, 2007.
  
  \bibitem{DavHauKra2021}
  A.~C. Davison, S.~Hautphenne, and A.~Kraus.
  \newblock Parameter estimation for discretely observed linear birth-and-death
  processes.
  \newblock {\em Biometrics}, 77(1):186--196, 2021.
  
  \bibitem{GoodmanMLEAccuracy2022}
  Jesse Goodman.
  \newblock Asymptotic accuracy of the saddlepoint approximation for maximum
  likelihood estimation.
  \newblock {\em Ann. Statist.}, 50(4):2021--2046, 2022.
  
  \bibitem{PedDavFok2015}
  Xanthi Pedeli, Anthony~C. Davison, and Konstantinos Fokianos.
  \newblock Likelihood estimation for the {INAR$(p)$} model by saddlepoint
  approximation.
  \newblock {\em Journal of the American Statistical Association},
  110(511):1229--1238, 2015.
  
\end{thebibliography}
\end{document}